\documentclass[reqno,11pt,a4paper]{amsart}

\usepackage{amsfonts,amsmath,amssymb,amsthm,color}

\usepackage{bbm}
\usepackage{amssymb}
\usepackage{amsfonts}
\usepackage{amsmath}
\usepackage{amsthm}
\usepackage{enumerate}
\usepackage{nccmath}
 \usepackage{mathrsfs}
\usepackage[USenglish]{babel}
\usepackage{esint}
\usepackage{graphicx}
\usepackage[colorlinks=true]{hyperref}
\usepackage[hmargin=3cm, vmargin=2.9cm]{geometry}
\hypersetup{
    colorlinks,
    citecolor=red,
    linkcolor=blue,
    urlcolor=black
}

\definecolor{gab}{RGB}{160, 48, 222}

\newcommand{\R}{\mathbb{R}}
\newcommand{\N}{\mathbb{N}}

\newcommand{\ph}{\varphi}
\newcommand{\eps}{\varepsilon}

\newtheorem{theorem}{Theorem}[section]
\newtheorem{lemma}[theorem]{Lemma}
\newtheorem{proposition}[theorem]{Proposition}

\newtheorem{remark}[theorem]{Remark}
\numberwithin{equation}{section}

\begin{document}
	\title[Non-degeneracy ]{Non-degeneracy of the bubble in a fractional and singular 1D Liouville equation}

\author{Azahara DelaTorre}
\address{Azahara DelaTorre, Dipartimento Matematica Guido Castelnuovo, Sapienza Universit\`a di Roma, Piazzale Aldo Moro 5, 00185 Roma (Italy)}
\email{azahara.delatorrepedraza@uniroma1.it}

\author{Gabriele Mancini}
\address{Gabriele Mancini, Dipartimento di Matematica, Università degli Studi di Bari Aldo Moro, Via Orabona 4, 70125 Bari (Italy)}
\email{gabriele.mancini@uniba.it}

 \author{Angela Pistoia}
\address{Angela Pistoia, Dipartimento di Scienze di Base e Applicate per l'Ingegneria, Sapienza Universit\`a di Roma, Via Antonio Scarpa 10, 00161 Roma (Italy)}
\email{angela.pistoia@uniroma1.it}

\author{Luigi Provenzano}
\address{Luigi Provenzano, Dipartimento di Scienze di Base e Applicate per l'Ingegneria, Sapienza Universit\`a di Roma, Via Antonio Scarpa 10, 00161 Roma (Italy)}
\email{luigi.provenzano@uniroma1.it}

\begin{abstract}
We prove the non-degeneracy of solutions to a fractional and singular Liouville equation defined on the whole real line in presence of a singular term. We use conformal transformations to rewrite the linearized equation as a Steklov eigenvalue problem posed in a bounded domain, which is defined either by an intersection or a union of two disks. We conclude by proving the simplicity of the corresponding eigenvalue.
\end{abstract}
\date\today
\subjclass{ 35R11 (35B33, 45G05)}
\keywords{Non-degeneracy, Liouville equation, Steklov problem}

\maketitle
\section{Introduction}
In this work we investigate non-degeneracy properties { of} solutions to the one-dimensional singular Liouville equation
\begin{equation}\label{eq:main}
	(-\Delta)^{\frac12} u =|x|^{\alpha-1}e^u \ \hbox{in}\ \mathbb R,
\end{equation}
with $0<\alpha <2$. In order to define the half-Laplacian in \eqref{eq:main}, we require  
\begin{equation}\label{L12}
\int_{\R} \frac{|u|}{1+x^2} <+\infty.
\end{equation}
We also assume the integrability condition 
\begin{equation}\label{CondInt}
\int_{\R} |x|^{\alpha-1} e^u < + \infty.
\end{equation}
Under conditions \eqref{L12} and \eqref{CondInt}, weak solutions to \eqref{eq:main}  are completely classified. When $\alpha = 1$, the set of solutions contains only the two-parameter family of solutions 
\begin{equation}\label{bub-non}
\mathfrak{u}_{\mu,\xi}(x)=\ln\left(\frac{2 \mu}{|x-\xi|^{2}+\mu^2}\right),
\end{equation}
with $\xi \in \R$ and $\mu>0$. We refer to the work of Da Lio, Martinazzi and Rivière in \cite{DMR15} for the proof. Due to translation and dilation invariance, it is clear that the derivatives
\begin{equation} \label{Zdefs}
{z}_{0, \mu, \xi}(x) := \partial_\mu \mathfrak{u}_{\mu, \xi}(x) = \frac{1}{\mu} \frac{(x - \xi)^2- \mu^2}{\mu^2 + (x - \xi)^2}, \quad z_{1, \mu, \xi}(x) := \partial_\xi \mathfrak{u}_{\mu, \xi}(x) = \frac{2 (x - \xi)}{\mu^2 + (x - \xi)^2}
\end{equation}
solve the linear problem 
\begin{equation}\label{eq:LinReg}
(-\Delta)^\frac{1}{2} z = e^{\mathfrak{u}_{\mu,\xi}}z \quad \text{ in }\R. 
\end{equation}
It is well known (see \cite{DdM05,S19,CF22}) that the bubble $\mathfrak{u}_{\mu, \xi}$ is non-degenerate up to the natural invariances of \eqref{eq:main}, i.e. the two functions in \eqref{Zdefs} span the space of all bounded solutions to \eqref{eq:LinReg}. More precisely, if $z \in L^\infty(\mathbb R)$ is a weak solution to  \eqref{eq:LinReg}, then $z$ is a linear combination of $z_{0, \mu, \xi}$ and $z_{1, \mu, \xi}$.

\medskip
If $\alpha \neq 1$, problem \eqref{eq:main} is not translation invariant. As we will show in Section \ref{Sec:prel}, it follows from the results  obtained by Gálvez, Jiménez and Mira in \cite{gjm} (see also \cite{ZZ18}) that for any $\alpha \in (0,1)\cup (1,2)$, equation \eqref{eq:main} only has a one-parameter family of solutions given by:   
\begin{equation}\label{bub-sing}u_\rho(x)= \ln\left(\frac{2\alpha \rho\sin\frac{\pi\alpha}2}{|x|^{2\alpha}+2\rho|x|^\alpha \cos\frac{\pi\alpha}{2}+\rho^2 }\right)
\end{equation}
with $\rho >0$.  We stress that the condition $\alpha\in (0,1)\cup (1,2)$ is necessary, since there exists no solution to \eqref{eq:main} when $\alpha \ge 2$ (see Proposition \ref{ClassificationU}).

In the present work we prove the non-degeneracy of $u_\rho$. Specifically, {for any $\alpha\in (0,1)\cup (1,2)$ }and $\rho>0$, we classify all solutions to the linearized problem \begin{equation}\label{eq:lin}
	(-\Delta)^{\frac12} \varphi =|x|^{\alpha-1}e^{u_\rho}\varphi \ \hbox{in}\ \mathbb \R,
\end{equation}
in the space of functions satisfying the conditions 
\begin{equation}\label{cond:phi}
	(-\Delta)^\frac{1}{4}\ph \in L^2(\R) \quad \text{ and } \quad  \int\limits_{ \mathbb R } |x|^{\alpha-1}e^{u_\rho } \varphi^2<+\infty.
\end{equation}
We consider the function 
\begin{equation}\label{z1d}
	z_{\rho} (x) := \partial_\rho u_\rho(x)  =\frac{1}{\rho} \, \frac{|x|^{2\alpha}-\rho^2}{|x|^{2\alpha}+2\rho |x|^{\alpha}\cos \frac{\pi \alpha}{2} + \rho^2},
\end{equation}
and give the following result:

\medskip
\begin{theorem}\label{Thm:1d} { Let $\alpha \in (0,1)\cup (1,2)$ and $\rho>0$. The function $u_{\rho}$ defined as in \eqref{bub-sing} is non degenerate.} That is,
if $\ph$ is a weak solution to \eqref{eq:lin} such that  \eqref{cond:phi} holds, then there exists $c\in \R$ such that $\ph = c \, z_{\rho}$, where $z_{\rho}$ is defined as in \eqref{z1d}. 
\end{theorem}

The main idea of the proof consists in proving the equivalence between the {\em non-local} eigenvalue problem  
 \begin{equation}\label{eig1}
		(-\Delta)^{\frac12} \varphi =\lambda |x|^{\alpha-1}e^{u_\rho}\varphi \ \hbox{in}\ \mathbb \R,
	\end{equation}
\medskip
and the {\em Steklov} eigenvalue problem
\begin{equation}\label{eig2}
\Delta \psi  =0 \text{ in } \Omega_\alpha,\; \partial_\nu \psi= \mu \psi  \text{ in } \partial\Omega_\alpha, \end{equation}

\noindent
where $\Omega_\alpha$ is either the intersection of two disks, when $\alpha\in(0,1)$, or the union of two disks, when $\alpha\in(1,2)$.
We will prove  that  the eigenvalue 
$\lambda=1$ of \eqref{eig1} corresponds to the eigenvalue $\mu_\alpha=\frac1{\sqrt{1+\tau_\alpha^2}}$ of \eqref{eig2}, being
$\tau_\alpha:=\frac{1+\cos\alpha\pi}{\sin\alpha\pi}$ and it is {always} simple {when $\alpha \in (0,1)\cup (1,2)$}.  It is worthwhile to point out that  {$\mu_\alpha$} is the {second} eigenvalue {of \eqref{eig2}} if  $\alpha\in(0,1),$ while it is the third eigenvalue when $\alpha\in(1,2).$ 
As a consequence, the Morse index of the bubble $u_\rho$ changes when $\alpha$ crosses the value  1. Indeed, it turns out to be equal to 1 when $\alpha\in(0,1),$  while it equals 2 when $\alpha\in(1,2)$.
\\

  The  proof of Theorem \ref{Thm:1d} is based on harmonic extension techniques (see \cite{CS}). Via convolution with the Poisson Kernel, every function satisfying \eqref{L12} can be extended  to a harmonic function defined on the upper half-plane $\R^2_+:=\{(x,y)\in \R^2\,:\; y>0\}$. It is simple to verify that the harmonic extensions of  \eqref{bub-non} and \eqref{bub-sing} are given respectively by 
\begin{equation*}%\label{bubblereg}
\mathcal U_{\mu,\xi}(x,y):=\ln\left(\frac{2\alpha \mu}{(x-\xi)^2+(y+\mu)^2}\right)
\end{equation*}
and by
\begin{equation}\label{bubble}
U_\rho(x,y):=\ln\frac{2\alpha \rho|\sin\theta_0|}{|z^\alpha-z_0|^2},\ z = x + i y, \ z_0=\rho e^{i\theta_0},\ \theta_0:= \frac{\pi\alpha}2+\pi.%,\ \rho>0.
\end{equation}
These functions solve the local problem
\begin{equation}\label{eq:ExtU}
-\Delta  U  = 0\ \hbox{in}\ \mathbb \R^2_+,\ 
\partial _\nu U= |x|^{\alpha-1} e^U \ \hbox{on}\ \partial \mathbb R^2_+,
\end{equation}
respectively for $\alpha =1$ and $\alpha\in (0,1)\cup(1,2)$, where $\nu$ is the outward normal to the half-plane $\partial \mathbb R^2_+$.  Similarly, if $\ph$ solves the \eqref{eq:lin}-\eqref{cond:phi}, then the harmonic extension $\Phi$ of $\ph$ satisfies
\begin{equation}\label{Lin}
-\Delta  \Phi  = 0\ \hbox{in}\ \mathbb R^2_+,\ 
\partial _\nu \Phi = |x|^{\alpha-1} e^U \Phi \ \hbox{on}\ \partial \mathbb R^2_+,
\end{equation}
as well as
\begin{equation}\label{weak}\int\limits_{  \mathbb R^2_+}|\nabla\Phi|^2+\int\limits_{\R^2_+} |z|^{2(\alpha-1)} e^{2U} \Phi^2\, dz + \int\limits_{\partial \mathbb R^2_+} |x|^{\alpha-1}e^U  \Phi^2<+\infty.\end{equation}

\begin{theorem}\label{main}
{For any $\alpha \in (0,1)\cup (1,2)$ and $\rho>0$, the function $U_\rho$ be defined as in \eqref{bubble} is non-degenerate.} Namely, each solution to the linear problem \eqref{Lin} satisfying \eqref{weak} is of the form 
\begin{equation}\label{eq:Phi_cond}
	\Phi(z)=\mathfrak{c} \frac{\partial U_\rho}{\partial\rho}(z),\ \mathfrak c\in\mathbb R ,\ \hbox{with}\ \frac{\partial U_\rho}{\partial\rho}(z)=\frac1\rho\frac{ |z|^{2\alpha}-|z_0|^2}{|z^\alpha - z_0|^2}.\end{equation}
\end{theorem}

In \cite{DdM05}, Dávila, del Pino and Musso studied  problem \eqref{Lin} with $\alpha=1$, and proved that it  is  equivalent to the study of the first nontrivial Steklov eigenspace for the unit disk $\mathscr{D}\subseteq \R^2$ (they use the fact that the half-plane is conformally equivalent to  $\mathscr{D}$). In  \cite{S19,CF22} Santra as well as Cozzi and Fernández directly attacked problem \eqref{eq:lin} and, using the stereographic projection of the real line on  $\mathbb S^1$, they wrote problem \eqref{eq:lin} with $\alpha =1$ as an eigenvalue problem of the fractional Laplacian on $\mathbb S^1.$  Neither of the approaches can be followed if $\alpha \neq 1$ because of the presence of the  non-autonomous term $|x|^{\alpha-1}$.
In the present paper, we find a clever change of variables which allows us to get rid of this term and to reduce the linear problem \eqref{Lin}  to a classical Steklov eigenvalue problem defined on a Lipschitz continuous bounded domain in the plane. 
 More precisely, we proceed as follows. First, using a conformal change of variables, we rewrite \eqref{Lin} on a cone (see \eqref{lin2}) so that the boundary condition does not contain the non-autonomous term anymore. Then, using a conformal M\"obius map, we rewrite \eqref{Lin}  as the Steklov eigenvalue problem (see \eqref{eig2}) with $\mu = \mu_\alpha$. The proof of Theorem \ref{main} is concluded provided that $\mu_\alpha$ is a simple eigenvalue. \\

The simplicity of the eigenvalue $\mu_{\alpha}$ for all $\alpha\in(0,1)\cup(1,2)$ is a delicate issue. By homothety, it is equivalent to the simplicity of the eigenvalue $1$ on all possible intersections ($\alpha\in(0,1)$) and unions ($\alpha\in(1,2)$) of two non-disjoint unit disks. In an appropriate coordinate system, in which the centers of the disks are symmetric and lie on the $y$ axis (see Figures \ref{figM1} and \ref{fig0}), the coordinate function $x$ is always an eigenfunction with eigenvalue $1$. These domains have two axes of symmetry, hence each eigenspace is spanned by even or odd functions with respect to each axis. Moreover, any second eigenfunction has exactly two nodal domains. From these symmetry considerations it follows that $x$ is a valid candidate for a second eigenfunction. This is in fact the case for the intersection of disks, where one expects that the nodal line of a second eigenfunction is the shortest possible. The rigorous proof requires a precise estimation of the eigenvalues associated with  eigenfunctions which are odd in $y$, which turn out to be strictly greater than $1$, ruling out also the possibility that $1$ is multiple. This is done by applying a Rellich-Pohozaev identity which allows to relate the Steklov eigenvalues with the eigenvalues of the  Dirichlet Laplacian on a part of the boundary, or, equivalently, on a segment. As for the union of disks, $1$ is no more the second eigenvalue by Weinstock's inequality (this is expected, now $x$ has not the shortest nodal line). Hence it is at least the third. This case is trickier: when the union of disks is close to the unit disk, by spectral stability $1$ is the third eigenvalue and it is simple. We need to prove that it can never be multiple. Again, we use a Rellich-Pohozaev identity to rule out the possibility of having third eigenfunctions with at least three nodal domains by estimating from below the corresponding eigenvalues. Note that the strict relation between the Steklov eigenvalues of a domain and the Laplace-Beltrami eigenvalues of its boundary through Rellich-Pohozaev identities has been object of a quite intense analysis, see e.g., \cite{PS}. In our specific situation, the boundary is one-dimensional, hence we get very precise information on Steklov eigenvalues associated with eigenfunctions satisfying certain symmetries.

\medskip

It is interesting to compare Theorem \ref{Thm:1d} with similar results in higher dimension. Equation \eqref{eq:main} is a one-dimensional analog of the celebrated Liouville equation 
\begin{equation}\label{Lio2d}
-\Delta u =  |x|^{2(\alpha-1)} e^{2u},
\end{equation}
which was introduced by Liouville \cite{Liou} with $\alpha =1$. Solutions to \eqref{Lio2d} with $|x|^{2(\alpha-1)}e^{2u}\in L^1(\R^2)$ were classified by Chen and Li \cite{ChenLi91} for $\alpha =1$, and by Prajapat and Tarantello  \cite{PrTar01} for a general $\alpha >0$. Non-degeneracy of solutions was proved by Baraket and Pacard in \cite{BarPac98} for $\alpha =1$, Esposito in \cite{Esposito05} for $\alpha \in (0,+\infty) \setminus \N$ and  del Pino, Esposito and Musso in \cite{DelPinoEspsitoMusso}, for $\alpha\in\N\setminus\{0\}$. We also quote the paper \cite{ggn}, where Gladiali, Grossi and Neves 
studied the Morse index of the solution of \eqref{Lio2d} showing that it changes  and increases whenever $\alpha$ crosses an integer value.  In recent years, Liouville equations have been studied also in dimension $n\ge 3$ in connection to problems involving higher order notions of curvature such as prescribed $Q$-curvature or prescribed fractional curvature problems (see e.g. \cite{DIEHyder,IMRNHMM,DGHM}). In particular in \cite{IMRNHMM}, Hyder, Mancini and Martinazzi consider the problem 
\begin{equation}\label{Eq:singDimn}
(-\Delta)^\frac{n}{2} u = |x|^{n(\alpha-1)} e^{n u} \quad \text{ in }\R^n
\end{equation}
with
$$
\int_{\R^n}|x|^{n(\alpha-1)} e^{n u}\, dx <+\infty.
$$
If $\alpha=1$, solutions satisfying  $u(x)=o(|x|^2)$ as $|x|\to \infty$ are completely classified (see \cite{Xu,MartClass} and non-degeneracy has been proved when $n$ is even (see \cite{BDOP07,Morlando}). However, there are also solutions to \eqref{Eq:singDimn} which behave at infinity as a quadratic polynomial (see \cite{Lin98,ChangChen01}). The singular case is more difficult to study. Differently from the 1d-case and 2d-case, if $n\ge 3$ and $\alpha \neq 1$, there is no explicit example of solution to \eqref{Eq:singDimn}. However, in \cite{IMRNHMM} it is proved that for any $\alpha>0$, \eqref{Eq:singDimn} has a radially symmetric solution with logarithmic behavior at infinity and  infinitely many radially symmetric solutions with polynomial behavior at infinity. To our knowledge no non-degeneracy result has been obtained so far. \\

We point out that the one-dimensional case that we treat in Theorem \ref{Thm:1d} is the only one in which a restriction on $\alpha$ appears. Moreover, Theorem \ref{Thm:1d} is the first classification result for the linearization of \eqref{Eq:singDimn} with $\alpha\neq 1$ in odd dimension, which makes \eqref{Eq:singDimn} non-local. Non-degeneracy results for non-local problems are extremely delicate to obtain. For sake of completeness we quote some results concerning the non-degeneracy of solutions in the fractional framework.
The non-degeneracy of solutions to the   non-local critical equation 
$$(-\Delta)^s u=u^\frac{n+2s}{n-2s} \ \hbox{in}\ \mathbb R^n$$
was studied by Chen, Frank and Weth and 
D\'avila, del Pino and Sire in \cite{cfw,DDS}. In the subcritical regime, i.e $1<p<\frac{n+2s}{n-2s}$, the non-degeneracy of least energy solution
$$(-\Delta)^s u+u=u^{p} \ \hbox{in}\ \mathbb R^n,$$
was completely achieved by Frank, Lenzmann and Silvestre in 
\cite{FLS}, after preliminary
works in particular cases discussed by
Fall and Valdinoci \cite{FV} when $s$ is close to 1 and by
Frank and Lenzmann \cite{FL} when $n=1$.
The non-degeneracy of minimizers for the fractional Caffarelli-Kohn-Nirenberg
inequality,  which after multiplication by $|x|^{-\alpha}$ are solutions to
 $$(-\Delta)^s u +\tau \frac u{|x|^{2s}}=|x|^{-(\beta-\alpha)p}u^{p-1} \ \hbox{in}\ \mathbb R^n$$
 with $p=\frac{2n}{n-2s+2(\beta-\alpha)}$, $\tau\ge 0$ and  
 $-2s<\alpha<\frac{n-2s}2$ and  $\alpha\le\beta<\alpha+s,$
 was obtained by Ao, DelaTorre and González in \cite{ADG} (see also \cite{DNGK}), while the non-degeneracy of 
 minimizers for the fractional Hardy-Sobolev inequality, namely solutions to  (i.e. $\tau=0$ and $\alpha=0$ in the previous equation)
 $$(-\Delta)^s u =|x|^{-\beta p}u^{p-1} \ \hbox{in}\ \mathbb R^n$$
was obtained by Musina and Nazarov in \cite{MN} and { to the critical fractional Hénon equation 
 $$(-\Delta)^s u =|x|^{\alpha}u^{\frac{N+2s+2\alpha}{N-2s}} \ \hbox{in}\ \mathbb R^n$$ by} Alarcon, Barrios and Quaas in \cite{ABQ}.
 
\medskip
The non-degeneracy result of Theorem \ref{Thm:1d} plays a role in the description of parameter-depending problems in which concentration phenomena occur and in which \eqref{eq:main} appears as a limit problem. For example, we refer to \cite{DDMW,DMP,AS1,AS2} for applications of \eqref{eq:main} and \eqref{eq:ExtU} to physical models for the description of galvanic corrosion phenomena for simple electrochemical systems (see e.g. \cite{Deconinck,VX}  and references therein). We believe that Theorems \ref{Thm:1d} and \ref{main} could be useful in the description of non-simple blow-up phenomena for such models.
\medskip

\medskip
This paper is organized as follows. In Section \ref{Sec:prel}, we introduce the notation and we recall some useful results. In Section \ref{Sec:change}, we introduce the changes of variable which allow to reduce Theorems \ref{Thm:1d} and \ref{main} to the study of a Steklov eigenvalue problem, which is studied in Section \ref{Sec:Steklov} concluding the proof.

\section{Preliminaries and classification results}\label{Sec:prel}
Throughout the paper we will denote 
$$
L_\frac{1}{2}(\R) :=\left\{u \in L^1_{loc}(\R)\;:\; \int_{\R} \frac{|u|}{1+x^2}<+\infty\right\}. 
$$
{If  $u\in L_\frac{1}{2}(\R)$, then it belongs to the dual of the Schwartz space $\mathcal  S$ of rapidly decreasing functions. Then, for any $s>0$ it is possible to define the fractional Laplacian} $(-\Delta)^s$ in the sense of tempered distribution by means of the Fourier Transform:
$$
<(-\Delta)^s u, \psi> = \int_{\R} u(\xi) (-\Delta)^{s} \psi \, d\xi, \quad \text{ where } (-\Delta)^{s} \psi = \mathcal{F}^{-1}[|\xi|^{2s} \mathcal F[\ph]], \quad {\psi\in\mathcal S}.
$$ 
{For a given function $f\in L^1_{loc}(\R)$, we say that $u\in L_{\frac{1}{2}}(\R)$} is a weak solution to $(-\Delta)^\frac{1}{2} u = f$ if 
$$
\int_{\R} u  (-\Delta)^\frac{1}{2}\psi  = \int_{\R} f \psi,
$$
for any $\psi\in C^\infty_c(\R)$. In particular, if $u\in L_\frac{1}{2}(\R)$ and \eqref{CondInt} holds, then we say that $u$ is a weak solution to \eqref{eq:main} if 
$$
\int_{\R} u  (-\Delta)^\frac{1}{2}\psi  = \int_{\R} |x|^{\alpha-1} e^u \psi
$$
for any $\psi \in C^\infty_c(\R)$.

We now state a result concerning regularity of weak solutions. We refer to \cite{IMRNHMM} for the proof.

\begin{lemma}\label{lemma:regularity}
Assume $\alpha \in (0,+\infty)$ and let $u\in L_\frac{1}{2}(\R)$ be a weak solution to \eqref{eq:main} such that \eqref{CondInt} holds. Then $u\in C^\infty(\R\setminus \{0\})\cap C^{0,\beta}_{loc}(\R)$ for some $\beta\in (0,1)$. 
\end{lemma}
 Condition \eqref{CondInt} also allows to describe the asymptotic behavior of $u$ as $|x|\to \infty$. 

\begin{lemma}\label{lemma:asym u}
Assume $\alpha\in (0,+\infty)$ and let $u\in L_\frac{1}{2}(\R)$ be a weak solution to \eqref{eq:main} such that \eqref{CondInt} holds. Then there exist $\beta >\alpha$ and $C>0$ such that 
$$
|u(x)  + \beta \ln|x||\le C,
$$
for all $x\in \R$ with $|x|\ge 1$. 
\end{lemma}
We refer again to \cite{IMRNHMM} for the proof. In fact, following the arguments of \cite{IMRNHMM} one can show that $\beta = 2\alpha$. However, for our purposes here we only need the estimate of $\beta > \alpha$. 

\medskip
To relate \eqref{eq:main}  with \eqref{eq:ExtU}, we let 
\begin{equation*}
P_y(x) = \frac{1}{\pi} \frac{y}{x^2+y^2}
\end{equation*} denote the Poisson kernel for the half-plane $\R^2_+:=\{(x,y)\in \R^2\,:\; y>0\}$. For a function $u\in L_\frac{1}{2}(\R)$, we can define the Poisson extension of u as 
\begin{equation*}
U(x,y) := (u * P_y)(x) = \frac{1}{\pi} \int_{\R} \frac{y\, u(\xi)}{(x-\xi)^2+y^2}\,d\xi, \quad (x,y)\in \R^2_+.
\end{equation*}

We recall the following standard properties of Poisson extensions:

\begin{proposition}\label{PropExt} Assume $u\in L_\frac{1}{2}(\R)$.
\begin{enumerate}
\item If $u\in C(a,b)$ for some $a,b\in \R$, $a<b$, then $U$ extends continuously to $(a,b)\times\{0\}$ and $U(x,0)=u(x)$ for any $x\in (a,b)$. 
\item If $u\in C^{1,s}(a,b)$ for some $s\in (0,1)$ and $a,b\in \R$ with $a<b$, then the partial derivatives of $U$ extend continuously to $(a,b)\times\{0\}$ and $\frac{\partial U}{\partial y}(x,0)=-(-\Delta)^\frac{1}{2}u(x)$ for any $x\in (a,b)$. 
\item If $(-\Delta)^\frac{1}{4}u\in L^2(\R)$, then $|\nabla U|\in L^2(\R^2_+)$ and $\|\nabla U\|_{L^2(\R^2_+)} = \|(-\Delta)^\frac{1}{4}u\|_{L^2(\R)}$. 
\end{enumerate}
\end{proposition}

\begin{lemma}\label{lemma:PropU}
Assume that $u\in L_\frac{1}{2}(\R)$ is a weak solution of \eqref{eq:main} such that \eqref{CondInt} holds. Then the harmonic extension $U= u*P_y$ is a {weak} solution to \eqref{eq:ExtU}. Moreover the following properties are satisfied.
\begin{enumerate}
\item\label{U1} 
$U\in C^\infty(\R^2_+)\cup C(\overline{\R^2_+})$ and 
\begin{equation}\label{IntUBoundary}
\int_{\partial \R^2_+}  |x|^{\alpha-1} e^U\,<+\infty.
\end{equation} 
\item\label{U2} Let $\beta$ be as in Lemma \ref{lemma:asym u}. Then, there exists $C>0$ such that 
\begin{equation}\label{LogBeta}
\left|U(x,y)+ \beta \ln\sqrt{x^2+y^2}\right|\le C,\quad \text{ in } \R^2_+\setminus B_2(0,0).
\end{equation}
In particular 
\begin{equation}\label{IntU}
\int_{\R^2_+}|x|^{2(\alpha-1)}e^{2U} <+\infty. 
\end{equation}
\end{enumerate}
\end{lemma}
\begin{proof}
	{By Lemma \ref{lemma:regularity} and Proposition \ref{PropExt} we can assert that $U\in C(\overline{\R^2_+})$. In particular $U(x,0)$ is continuous. Moreover, Lemma \ref{lemma:asym u} guarantees that $ |x|^{\alpha-1} e^{U(x,0)}= O \left(\frac{1}{|x|^{1+\beta-\alpha}}\right)$ for $|x|\ge 1$, for some $\beta>\alpha$. Thus we get \eqref{U1}}.
	
Now we prove \eqref{U2}, which is a consequence of the formula 
\begin{equation}\label{extLog}
\frac{1}{\pi}\int_\R \frac{y \ln |\xi|}{(x-\xi)^2+y^2}\,d\xi = \ln \sqrt{x^2+y^2}. 
\end{equation}
Indeed, \eqref{extLog} gives 
$$
\left|U(x,y)+ \beta \ln\sqrt{x^2+y^2}\right| \le \frac{1}{\pi}\int_{\R} \frac{y |u(\xi)+\beta \ln|\xi||}{(x-\xi)^2+y^2}\,d\xi \le C + \frac{1}{\pi}\int_{-1}^1  \frac{y |u(\xi)+\beta \ln |\xi||}{(x-\xi)^2+y^2}\,d\xi,
$$
where the last inequality follows from Lemma \ref{lemma:asym u}. Finally, we observe that if $\sqrt{x^2+y^2}\ge 2$,  then 
$$
\sqrt{(x-\xi)^2+y^2} = |(x,y)-(\xi,0)|\ge  |(x,y)|-|\xi|\ge \frac{|(x,y)|}{2},
$$
for any $\xi \in [-1,1]$, so that 
$$
\frac{y}{(x-\xi)^2+y^2}\le \frac{4y}{x^2+y^2}\le \frac{4}{\sqrt{x^2+y^2}}\le 2.
$$
Then 
$$
\frac{1}{\pi}\int_{-1}^1  \frac{y|u(\xi)+\beta \ln |\xi||}{(x-\xi)^2+y^2}\,d\xi\le \frac{2}{\pi}\| u\|_{L^1(-1,1)} + \frac{2}{\pi}\int_{-1}^1|\ln |\xi||\,d\xi =   \frac{2}{\pi}\| u\|_{L^1(-1,1)} + \frac{4}{\pi}.
$$
This proves \eqref{LogBeta}. Since $\beta >\alpha$, {again by Proposition \ref{PropExt}} we get \eqref{IntU}.
\end{proof}

\begin{proposition}\label{ClassificationU} Assume $\alpha\in (0,1)\cup (1,2)$. Let $U\in C(\overline{\R^2_+})$ be a {weak} solution to \eqref{eq:ExtU} such that \eqref{IntUBoundary} and \eqref{IntU} hold. Then, there exists $\rho >0$ such that $U = U_\rho$ where $U_\rho$ is defined as in \eqref{bubble}. Moreover, if $\alpha \ge 2$, there is no solution  to \eqref{eq:ExtU} which is continuous in $\overline{\R^2_+}$.  
\end{proposition}
\begin{proof}
Taking 
$$V(x,y)=2U(x,y)+2(\alpha-1)\ln\sqrt{x^2+y^2} $$
we see that $V$ solves 
\begin{equation}\label{gjm-0}
 \Delta  V  = 0\ \hbox{in}\ \mathbb R^2_+ ,\ 
\partial _\nu V= 2 e^\frac V2 \ \hbox{on}\ \partial \mathbb R^2_+\setminus \{0\},
\end{equation}  with  
\begin{equation}\label{gjm-1}
\int\limits_{\mathbb R^2_+}   e^{V(z)}dz<+\infty \quad \text{ and } \int\limits_{\mathbb \partial \R^2_+}   e^{\frac{V(z)}{2}}dz<+\infty.
\end{equation}
{Note that by Proposition \ref{PropExt} we can assert that \eqref{gjm-0} is satisfied in the classical sense.}
Since $U$ is continuous at $0$, we further have 
\begin{equation}\label{Vlog}
\lim_{(x,y)\to (0,0)} V(x,y)-2(\alpha-1)\ln \sqrt{x^2+y^2} = U(0,0).
\end{equation}
In \cite{gjm} it is proved that all the solutions  to \eqref{gjm-0}-\eqref{gjm-1} can be written in complex variable as
\begin{equation}\label{gjm-2}
V(z)=\ln\frac{4\lambda^2\gamma^2|z|^{2(\gamma-1)}}{|z^\gamma-z_0|^4},\ z_0=\rho e^{i\theta_0}.
\end{equation}
where $\rho>0$, $\gamma>0$ and the parameters $\lambda>0$ and $\theta_0$ must satisfy
\begin{equation}\label{gjm-3}
\lambda=-\rho\sin\theta_0\ \hbox{and}\ \lambda=\rho\sin(\theta_0-\pi\gamma),\end{equation}
or 
\begin{equation*}%\label{gjm-bis}
V(z) = \ln \frac{\pi^2}{|z|^2 |\ln z- z_0|},
\end{equation*}
where $z_0\in \mathbb C$ and $Im(z_0) = \frac{\pi}{2}$. Since  \eqref{Vlog} holds, we must have that \eqref{gjm-2}  hold and $\gamma = \alpha$. Furthermore, in order to have $V$ well defined on $\overline{\R^2_+}$, it is necessary that $\alpha = \gamma \in (0,2)$. Since we are also assuming $\alpha \neq 1$,  \eqref{gjm-3} yields
\begin{equation*}%\label{gjm-4}
  \theta_0=\frac{\pi\alpha}2+\pi.\end{equation*}
%  \hbox{and}\ \gamma\in(0,2)
Then we have proved that $U$ is given by 
$$
U(x,y) = \frac{1}{2}\ln \left(\frac{4 \alpha^2 \rho^2 \sin^2\theta_0}{|z^\alpha-\rho e^{i\theta_0}|^4}  \right) = U_\rho(x,y)
$$
for any $(x,y)\in \R^2_+$.
\end{proof}

As a straightforward consequence of Proposition \ref{ClassificationU}  and Lemma \ref{lemma:PropU} we get the following classification result for \eqref{eq:main}.

\begin{proposition}
Assume $\alpha \in (0,1)\cup (1,2)$ and let $u\in L_\frac{1}{2}(\R)$ be a weak solution to \eqref{eq:main} such that \eqref{CondInt} holds. Then there exists $\rho >0 $ such that $u = u_\rho$ where $u_\rho$ is defined as in \eqref{bub-sing}. 
\end{proposition}

 We now briefly discuss the equivalence of the linarized problems \eqref{eq:lin} and \eqref{Lin}.  Let us fix $\rho >0$ and $\alpha \in (0,1)\cup (1,2)$. To simplify the notation, in the following we write $u=u_\rho$, without writing explicitly the dependence on $\rho$. We consider the space 
 $$
 \mathcal{H}:=\{{\ph }\in L^1_{loc}(\R) \;:\; |x|^{\alpha-1} e^u \ph^2 \in L^1(\R),\;  (-\Delta)^\frac{1}{4}\ph \in L^2(\R)\}. 
 $$ 
 We observe that the condition $|x|^{\alpha
 -1} e^u \ph^2 \in L^1(\R)$ implies $\ph \in L_\frac{1}{2}(\R)$. 
 Indeed, we have
 $$  \int\limits_{\mathbb R} \frac{|\varphi(x)|}{1+x^2} = \int\limits_{\mathbb R} |x|^{\frac{\alpha-1}{2}}e^\frac u2 \frac{1}{|x|^{\frac{\alpha-1}{2}}e^\frac u2 }\frac{|\varphi(x)|}{1+x^2} \le\left(\int\limits_{\mathbb R} |x|^{\alpha-1}e^u\varphi^2\right)^\frac12
\left(\int\limits_{\mathbb R} {\frac{1}{|x|^{\alpha-1}e^u}}\frac1{(1+x^2)^2} \right)^\frac12 \hspace{-0.2cm},
$$
where, since $|x|^{\alpha-1}e^{u} \sim \frac{1}{|x|^{\alpha+1}}$ as $|x|\to +\infty$,  $|x|^{\alpha-1}e^u \sim  |x|^{\alpha-1}$ as $|x|\to 0$, and $\alpha\in (0,2)$,
$$
\int\limits_{\mathbb R} {\frac{1}{|x|^{\alpha-1}e^u} }\frac1{(1+x^2)^2} \le C\left(\int_{-1}^1\frac{1}{|x|^{\alpha-1}} + \int_{|x|\ge 1} \frac{1}{|x|^{3-\alpha}} \right)<+\infty. 
$$
We now show that $\ph$ can grow at most logarithmically as $|x|\to +\infty$. 

\begin{lemma}\label{lemma:grphi}
Assume that $\ph \in \mathcal{H}$ is a weak solution to \eqref{eq:lin}. Then, there exists $C_1,C_2 >0$ such that 
$|\ph(x)|\le C_1 + C_2 \ln|x|$ for any $x \in \R$ with $|x|>1$. 
\end{lemma}
\begin{proof}
We define $$
v(x) =\frac{1}{\pi}\int_{\R} \ln\left( \frac{1+|y|}{|x-y|}\right) |y|^{\alpha-1} e^u \ph(y) dy.  
$$
{We first note that $|y|^{\alpha-1}e^u\varphi\in L^1(\R)$. Indeed, by Hölder inequality,
\begin{equation*}
\int_{\R}	 |y|^{\alpha-1}e^u\varphi\,dy\leq \left(\int_{\R}|y|^{\alpha-1}e^u\,dy\right)^{1/2}\left(\int_{\R}|y|^{\alpha-1}e^u\varphi^2\,dy\right)^{1/2}<+\infty.
\end{equation*}
By \cite[Lemma 2.3]{DIEHyder}, $v\in L_\frac{1}{2}(\R)$ and it is a weak solution to $(-\Delta)^\frac{1}{2} v =|x|^{\alpha-1} e^u \ph$. } 

{Now we estimate
	$$
	|v(x)|\le \frac{1}{\pi}\int_{\R} \ln (1+|y|) |y|^{\alpha-1} e^{u(y)} |\ph (y)|\,dy + \frac{1}{\pi}\int_{\R} |\ln |x-y|| |y|^{\alpha-1} e^{u(y)} |\ph (y)|\,dy. 
	$$}
%$$
%|v(x)|\le |C| + \frac{1}{\pi}\int_{\R} \ln (1+y) |y|^{\alpha-1} e^{u(y)} \ph (y)\,dy + \frac{1}{\pi}\int_{\R} |\ln |x-y|| |y|^{\alpha-1} e^{u(y)} |\ph (y)|\,dy. 
%$$
By Holder's inequality we have
$$
\begin{aligned}
{\int_{\R} \ln (1+|y|) |y|^{\alpha-1}} &  e^{u(y)} |\ph (y)|\,dy \\
& {\le  \left(\int_{\R} \ln^2 (1+|y|) |y|^{\alpha-1} e^{u(y)}\, dy\right)^\frac{1}{2}} \left(\int_{\R}  |y|^{\alpha-1} e^{u(y)}\ph^2(y)\, dy\right)^\frac{1}{2}{ \leq C}.
\end{aligned}
$$
Moreover, {since $|x|>1$,}
$$\begin{aligned}
\int_{|x-y|\ge 1} \hspace{-0.2cm} & \ln |x-y| |y|^{\alpha-1} e^{u(y)} |\ph (y)|\,dy \\
& \le \int_{|x-y|\ge 1} \ln |x| |y|^{\alpha-1} e^{u(y)} |\ph (y)|\,dy + \int_{|x-y|\ge 1} \ln \left(1+\frac{|y|}{|x|}\right)) |y|^{\alpha-1} e^{u(y)} |\ph (y)|\,dy \\
& \le \ln |x| \int_{\R} |y|^{\alpha-1} e^{u(y)} |\ph (y)|\,dy + \int_{\R} \ln \left(1+|y|\right) |y|^{\alpha-1} e^{u(y)} |\ph (y)|\,dy \le C{(1+\ln |x| )},
\end{aligned}
$$
and 
$$\begin{aligned}
\int_{|x-y|\le 1} & \left|\ln |x-y|\right| |y|^{\alpha-1} e^{u(y)} |\ph (y)|\,dy \\
& \le \left( \int_{|x-y|\le 1}  \ln^2 |x-y| |y|^{\alpha-1} e^{u(y)} \,dy \right)^\frac{1}{2}\left( \int_{|x-y|\le 1}  |y|^{\alpha-1} e^{u(y)}\ph^2(y) \,dy \right)^\frac{1}{2} \\
& \le C \left( \int_{|x-y|\le 1}  \ln^2 |x-y|\,dy\right)^\frac{1}{2}\le C. 
\end{aligned}
$$
We can so conclude that there exist $C_1,C_2 >0$ such that
\begin{equation}\label{growthv}
|v(x)|\le  C_1 + C_2\ln |x|.
\end{equation}
 Since $h = \ph - v$ is $\frac{1}{2}-$harmonic in $\R$ and $h \in L_\frac{1}{2}(\R)$, by Liouville's theorem for the fractional Laplacian (see e.g  \cite[Theorem 4.4]{FallWeth}), we find that $h$ is constant. Then the conclusion follows by \eqref{growthv}. 
\end{proof}

\begin{lemma}\label{lemma:grPhi}
Assume that $\ph\in \mathcal{H}$ and let $\Phi$ be the harmonic extension of $\ph$. Then, there exists $C_1,C_2>0$ such that
$$
|\Phi(x,y)|\le C_1 + C_2 \ln(x^2+y^2) \quad \forall (x,y)\in \R^2_+ \setminus B_1((0,0)). 
$$
\end{lemma}
\begin{proof}
Indeed, by Lemma \ref{lemma:grphi} we know that
$$
-C_1-C_2 \ln |\xi| \le \phi(\xi) \le C_1 + C_2 \ln |\xi|,
$$
for any $\xi \in \R$ with $|\xi|\ge 1$. 
Then, the conclusion follows by formula \eqref{extLog}.
\end{proof}

Let us now consider the linearized problem \eqref{Lin}. We consider the space 
$$
\mathcal{H}(\R^2_+):=\left\{ Z \in L^1_{loc}(\R^2_+)\;:\; |\nabla Z|\in L^2(\R^2_+) \text{ and } |x|^{2(\alpha-1)} e^U Z^2\in L^1(\R^2_+)  \right\}.
$$
We say that a function $Z \in \mathcal{H}(\R^2_+)$ is a weak solution to \eqref{Lin} if
\begin{equation}\label{weakR2}
\int_{\R^2_+} \nabla Z \cdot \nabla \chi  = \int_{\partial\R^2_+} |x|^{\alpha-1} e^U Z \chi, \quad \forall \chi \in C^\infty_c(\overline{\R^2_+}). 
\end{equation}

\begin{remark}\label{testfunctions} A function $Z \in \mathcal{H}(\R^2_+)$ is a weak solution to \eqref{Lin} iff
\begin{equation}\label{weak_ex}
\int_{\R^2_+} \nabla Z \cdot \nabla \chi  = \int_{\partial\R^2_+} |x|^{\alpha-1} e^U Z \chi, \quad \forall \chi \in \mathcal{H}(\R^2_+){\cap L^\infty(\R^2_+)}.   
\end{equation}
\end{remark}
\begin{proof}
Assume that $\chi \in \mathcal{H}(\R^2_+)\cap L^\infty(\R^2_+)$. Let $\eta$ be a cut-off function such that $\eta\equiv 1$ for $|x|\le 1$, $\eta  \in C^\infty_c(B_2(0,0))$,  and $0\le \eta\le 1$. Given $R>0$, consider the functions $\eta_R(x) = \eta(\frac{x}{R})$ and $\chi_R(x) = \chi(x)\eta_R(x)$. Then, $\chi_R \in H^1(B_{2R}(0,0)\cap \R^2_+)$.  A standard density argument (see e.g. \cite[Theorem 3.22]{Adams}) shows that there exists a sequence of functions $\psi_n \in C^\infty_c(\R^2)$ such that $\psi_n \to \chi_R$ in $H^1(B_{2R}(0,0)\cap \R^2_+)$. For any $n\in \N$, we have the identity
$$
\int_{\R^2_+}   \nabla Z \cdot \nabla \psi_n  = \int_{\partial \R^2_+} |x|^{\alpha-1} e^U Z \psi_n .
$$
Using that $\psi_n \to \chi_R$ in $H^1(B_{2R}(0,0)\cap \R^2_+)$ , we easily get 
\begin{equation*}%\label{intermediate}
\int_{\R^2_+}   \nabla Z \cdot \nabla \chi_R = \int_{\partial \R^2_+} |x|^{\alpha-1} e^U Z \chi_R.
\end{equation*}
By dominated convergence, we have that 
$$
\int_{\partial \R^2_+} |x|^{\alpha-1} e^U Z \chi_R \to \int_{\partial \R^2_+} |x|^{\alpha-1} e^U Z \chi, \quad \text{ as }R\to +\infty.
$$
Moreover, we have that 
$$
\int_{\R^2_+} \nabla Z \cdot \nabla \chi_R  = \int_{ \R^2_+}  (\nabla Z \cdot \nabla \chi  )\eta_R + (\nabla Z \cdot \nabla \eta_R )\chi, \, 
$$
with 
$$\int_{ \R^2_+}  (\nabla Z \cdot \nabla \chi  )\eta_R \to \int_{\R^2_+} \nabla Z \cdot \nabla \chi  , \quad \text{ as }R\to +\infty,
$$
by dominated convergence, and
$$\begin{aligned}
\left|\int_{\R^2_+} (\nabla Z \cdot \nabla \eta_R ) \chi\right|& \le \|\chi\|_{L^\infty(\R^2_+)} \|\nabla \eta\|_{L^\infty(\R^2)}\frac{1}{R} \int_{\R^2_+ \cap( B_{2R}(0,0)\setminus B_R(0,0))} |\nabla Z| \\
& \le \sqrt{\frac{3}{2}\pi} \, \|\chi\|_{L^\infty(\R^2_+)} \|\nabla \eta\|_{L^\infty(\R^2)}\left( \int_{\R^2_+ \cap (B_{2R}(0,0)\setminus B_R(0,0))} |\nabla Z|^2 \right)^\frac{1}{2}  \to 0,
\end{aligned}
$$
as $R\to +\infty$. 
\end{proof}

\begin{remark} Using the changes of variables given in Section \ref{Sec:change}, one can actually show that for any $\chi \in \mathcal{H}(\R^2_+)$, $\int_{\partial \R^2_+}|x|^{\alpha-1} e^u \chi^2 <\infty$. Moreover, a function $Z\in \mathcal{H}(\R^2_+)$ is a weak solution to \eqref{weakR2} iff \eqref{weak_ex} holds for any $\chi$ in $\mathcal{H}(\R^2_+)$.
\end{remark}

\begin{proposition}\label{Prop:phi-Phi}
 Assume that $\ph \in \mathcal H$ {is a weak solution to \eqref{eq:lin}} and let $\Phi$ be the harmonic extension of $\ph$. Then $\Phi$ is a weak solution to \eqref{Lin}, {and \eqref{weak}} holds.
\end{proposition}

 \begin{proof}
By Proposition \ref{PropExt}, we know that 
$$
\int_{\R^2_+} \nabla \Phi \cdot \nabla \chi  = \int_{\partial\R^2_+} |x|^{\alpha-1} e^U \Phi \chi, \quad \forall \chi \in C^\infty_c(\overline{\R^2_+}\setminus\{(0,0)\}).
$$
We use a cut-off argument to show that \eqref{weakR2} holds.  Fix $\chi \in C^\infty_c(\overline{\R^2_+})$, and let $\eta \in C^\infty_c(B_2(0,0))$ be such that $0\le \eta \le 1$ and $\eta \equiv 1 $ on $B_1(0,0)$. For any $\eps >0 $ we denote $\eta_\eps (x) = (1-\eta(\frac{x}{\eps}))$. If $\chi \in C^\infty_c(\overline{\R^2_+})$, then  $\chi_\eps:=\chi \eta_\eps\in C^\infty_c(\overline{\R^2_+}\setminus \{0\})$. Then, for any $\eps >0$, we have
$$
\int_{\R^2_+} \nabla \Phi \cdot \nabla \chi_\eps   = \int_{\partial\R^2_+} |x|^{\alpha-1} e^U \Phi \chi_\eps\,.
$$
Noting that $\nabla \chi_\eps = \eta_\eps \nabla \chi + \chi \nabla \eta_\eps $ and that 
$$
\begin{aligned}
\left|\int_{\R^2_+} (\nabla \Phi \cdot \nabla \eta_\eps)\chi  \right| &  \le  \|\chi\|_{L^\infty(\R^2_+)} \|\nabla \eta\|_{L^\infty(\R^2)}  \frac{1}{\eps} \int_{\R^2_+\cap  (B_{2\eps}(0,0)\setminus B_{\eps}(0,0))} |\nabla \Phi| \\
&\le  \|\chi\|_{L^\infty(\R^2_+)} \|\nabla \eta\|_{L^\infty(\R^2)}  \sqrt{\frac{3}{2}\pi}\left( \int_{\R^2_+\cap (B_{2\eps}(0,0)\setminus B_{\eps}(0,0))} |\nabla \Phi|^2\right)^\frac{1}{2} \to 0,
\end{aligned}
$$
as $\eps \to 0$, we can use dominated convergence theorem, since $\eta_\eps \to 1$  pointwise on $\overline{\R^2_+}\setminus\{(0,0)\}$, to get 
$$
\int_{\R^2_+} \nabla \Phi \cdot \nabla \chi    = \int_{\partial\R^2_+} |x|^{\alpha-1} e^U \Phi \chi\,.
$$
Now \eqref{weak_ex} follows by Remark \ref{testfunctions}. Finally, we observe that \eqref{weak} is a consequence of Lemma \ref{lemma:grPhi} and Proposition \ref{PropExt}.
\end{proof}

\begin{remark}
Proposition \ref{Prop:phi-Phi} shows, in particular, that Theorem \ref{Thm:1d} follows by Theorem \ref{main}. 
\end{remark}

\section{Proof of Theorem \ref{main}}\label{Sec:change}
In this section, we will transform our problem into an equivalent one via conformal transformations. First, from the upper half-space to a cone, and then to a bounded domain which will be determined by an intersection or union of balls depending on the values of $\alpha\in(0,1)\cup(1,2).$ In the cone we will obtain a linear problem with Neumann-type boundary conditions which, in the bounded domain, will become a Steklov eigenvalue problem. This will allow us to prove the main result of the paper.

For sake of simplicity we rewrite \eqref{bubble} as
\begin{equation*}%\label{bubble2}
U (z)=\ln\frac{2\alpha |\xi_2|}{|z^\alpha-\xi|^2},\ \hbox{with}\ \xi=\xi_1+ i \xi_2,\ \xi_2<0,\ \frac{\xi_1}{\xi_2}={1+\cos\alpha\pi\over\sin\alpha\pi}.\end{equation*}

\subsection{An equivalent problem on a cone}
 Let us  consider the cone
\begin{equation}\label{eq:cone}\mathcal C_\alpha:=\left\{ (r\cos  \theta,r \sin  \theta)\ :\ r>0,\ \theta \in [0,\pi\alpha)\right\}.\end{equation}
Let $F_\alpha: \R^2_+ \rightarrow \mathcal C_\alpha$ be the complex power $ z^\alpha ,$ which using polar coordinates is written as
$$F_\alpha(x,y):=(r^\alpha\cos \alpha\theta,r^\alpha\sin\alpha\theta),\ x=r\cos\theta,\ y=r\sin\theta.$$
It is known that $F_\alpha$ is a  conformal  diffeomorphism between $\R^2_+$ and $\mathcal C_\alpha$.
A straightforward computation shows that a function $\Phi$ solves the linear problem \eqref{Lin}   if and only if the function $\phi = \Phi \circ F_\alpha^{-1}$ solves the linear problem 
\begin{equation}\label{lin2}-\Delta \phi =0\ \hbox{in}\ \mathcal C_\alpha,\ \partial_\nu \phi  = e^{W_\alpha}\phi \ \hbox{on}\ \partial\mathcal C_\alpha,
\end{equation}
where 
\begin{equation}\label{wa}W_\alpha (x,y):= U \left( F_\alpha^{-1}(x,y)\right) -\ln \alpha= \ln  \frac{2|\xi_2|}{(x-\xi_1)^2+(y-\xi_2)^2}.
\end{equation} 
In fact, for sake of completeness we give a  brief proof of the claim.
If $\Phi$ solves \eqref{Lin}   for any $\chi \in \mathcal{H}(\R^2_+)\cap L^\infty(\R^2_+)$, we have  
$$ 
\int\limits_{\mathbb R^2_+}\nabla \Phi\cdot \nabla\chi = \int\limits_{\partial \mathbb R^2_+} |x|^{\alpha-1}e^U \Phi \chi.$$
First, we point out that
 $$\begin{aligned} & \int\limits_{\mathbb R^2_+}\nabla \Phi\cdot\nabla\chi dxdy\\ & \hbox{ (setting $x=r\cos\theta,\ y=r\sin\theta$,  $\hat\Phi(r,\theta):=\Phi(r\cos\theta ,r\sin\theta )$ and $\hat\chi(r,\theta):=\chi(r\cos\theta ,r\sin\theta )$)}\\
 &
 =\int\limits_0^\infty  \int\limits_0^{\pi}\left(r\partial_r\hat\Phi\partial_r\hat\chi+\frac1 r\partial_\theta\hat\Phi\partial_\theta\hat\chi\right)drd\theta\\ & \hbox{ (setting $\rho=r^\alpha,$   $\gamma=\alpha\theta$,   $\tilde\Phi(\rho,\gamma):=\hat\Phi\left(\rho^\frac1\alpha,\frac\gamma\alpha \right)$ and $\tilde\chi(\rho,\gamma):=\hat\chi\left(\rho^\frac1\alpha,\frac\gamma\alpha \right)$})\\ &
 =\int\limits_0^\infty  \int\limits_0^{\alpha\pi}\left(\rho\partial_\rho\tilde\Phi\partial_\rho\tilde\chi+\frac1 \rho\partial_\gamma\tilde\Phi\partial_\gamma\tilde\chi\right)d\rho d\gamma \\ & \hbox{ (setting $s=\rho\cos\gamma,\ t=\rho\cos\gamma$,  $\phi(s,t):=\tilde\Phi(\rho,\gamma)$ and $\upsilon(x,y):=\tilde\chi(\rho,\gamma)$)}\\
&=\int\limits_{\mathcal C_\alpha}\nabla \phi\nabla \upsilon dsdt.\end{aligned}
 $$
Next, since    on $\partial\mathbb R^2_+$  
$$e^{U(x,0)}=\frac{2\alpha |\xi_2| }{ (x^\alpha-\xi_1)^2+|\xi_2|^2},\quad \ \hbox{if}\ x\ge0,$$
and
$$
e^{U(x,0)}=\frac{2\alpha |\xi_2| }{ (|x|^\alpha\cos\alpha\pi-\xi_1)^2+(|x|^\alpha\sin\alpha\pi-\xi_2)^2},\quad \ \hbox{if}\ x\le0,$$
we also have
$$\begin{aligned}&\int\limits_{\partial \mathbb R^2_+} |x|^{\alpha-1}e^U \Phi \chi\\ &=\int\limits_{\{(x,0):x\ge0\}}\frac{2\alpha |\xi_2||x|^{\alpha-1}}{ (|x|^\alpha-\xi_1)^2+|\xi_2|^2}\Phi\chi+\int\limits_{\{(x,0):x\le0\}}\frac{2\alpha |\xi_2||x|^{\alpha-1}}{ (|x|^\alpha\cos\alpha\pi-\xi_1)^2+|\xi_2|^2}\Phi\chi\\
&=\int\limits_0^\infty \frac{2  |\xi_2|}{ (\sigma-\xi_1)^2+|\xi_2|^2}\phi\upsilon d\sigma  +
 \int\limits_0^\infty \frac{2  |\xi_2|}{ (\sigma\cos\alpha\pi-\xi_1)^2+(\sigma\sin\alpha\pi-\xi_2)^2}\phi\upsilon d\sigma \\
 &=\int\limits_{\partial_-\mathcal C_\alpha}e^{W_\alpha(s,t)}\phi\upsilon+\int\limits_{\partial_+\mathcal C_\alpha}e^{W_\alpha(s,t)}\phi\upsilon=\int\limits_{\partial\mathcal C_\alpha}e^{W_\alpha(s,t)}\phi\upsilon,
 \end{aligned} $$
 because 
 \begin{equation}\label{bordoc}\partial\mathcal C_\alpha:=\underbrace{\{(\sigma,0)\ :\ \sigma\ge0\}}_{:=\partial_-\mathcal C_\alpha  }\cup\underbrace{\{(\sigma\cos\pi\alpha,\sigma\sin\pi\alpha)\ :\  \sigma\ge0\}}_{:=\partial_+\mathcal C_\alpha  }.\end{equation}
 Finally, we deduce that for any $\upsilon \in L^{\infty}(\mathcal C_\alpha)$ and thus, such that $\int\limits_{\partial \mathcal C_\alpha}  e^W \upsilon^2<+\infty$, if $v$ satisfies 
\begin{equation}\label{weak1}\int\limits_{\mathcal C_\alpha}|\nabla \upsilon|^2<+\infty,
\end{equation}
then,
$$\begin{aligned}
0=\int\limits_{\mathcal C_\alpha}\nabla \phi\nabla\upsilon-\int\limits_{\partial\mathcal C_\alpha}e^{W_\alpha(s,0)}\phi\upsilon,
\end{aligned} $$
that is, $\phi$ solves \eqref{lin2}.

\subsection{An equivalent problem on a bounded domain}

Let us consider the map 
\begin{equation}\label{eq:G}
G_\alpha(x,y) = \left(\frac{|\xi|^2- x^2-y^2} {(x-\xi_1)^2+(y-\xi_2)^2}, \frac{2(y \xi_1-  x\xi_2) }{(x -\xi_1)^2+(y-\xi_2)^2} \right),\
(x,y)\in\mathbb R^2\setminus\{\xi\}.\end{equation}

\begin{lemma}\label{LemmaPG}
	The function $G_\alpha$ given by \eqref{eq:G} satisfies the following properties:
\begin{enumerate}
\item\label{PCConf} $G$ is a conformal diffeomorphism between $\R^2\setminus\{\xi\}$ and $\R^2\setminus\{(-1,0)\}$.
\item\label{PGJacobian} The Jacobian of $G$ is given by (see \eqref{wa})
\begin{equation}\label{jaco}
 JG_\alpha(x,y) = \frac{4 |\xi|^2}{\left((x -\xi_1)^2+(y-\xi_2)^2\right)^2}  =\frac{|\xi|^2}{|\xi_2|^2} e^{2 W_\alpha(x,y)}.
\end{equation}
\item \label{cone} The image of the cone \eqref{eq:cone}, i.e., $G_\alpha(\mathcal C_\alpha)$ is
\begin{equation*}%\label{o+}
\Omega_\alpha:=\mathscr D_\alpha^-\cap\mathscr D_\alpha^+\ \hbox{if}\ \alpha\in(0,1)\quad \hbox{or}\quad
\Omega_\alpha:=\mathscr D_\alpha^-\cup\mathscr D_\alpha^+\ \hbox{if}\ \alpha\in(1,2),
\end{equation*}
where
\begin{equation}\label{circ}\mathscr D_\alpha^\pm:=\left\{(x,y)\in\mathbb R^2\ :\ x^2+(y\pm \tau_\alpha)^2\le 1+\tau_\alpha^2\right\},\
 \tau_\alpha:={1+\cos\alpha\pi\over\sin\alpha\pi}.\end{equation}
\end{enumerate}
\end{lemma}
\begin{center}\includegraphics[width=0.7\textwidth]{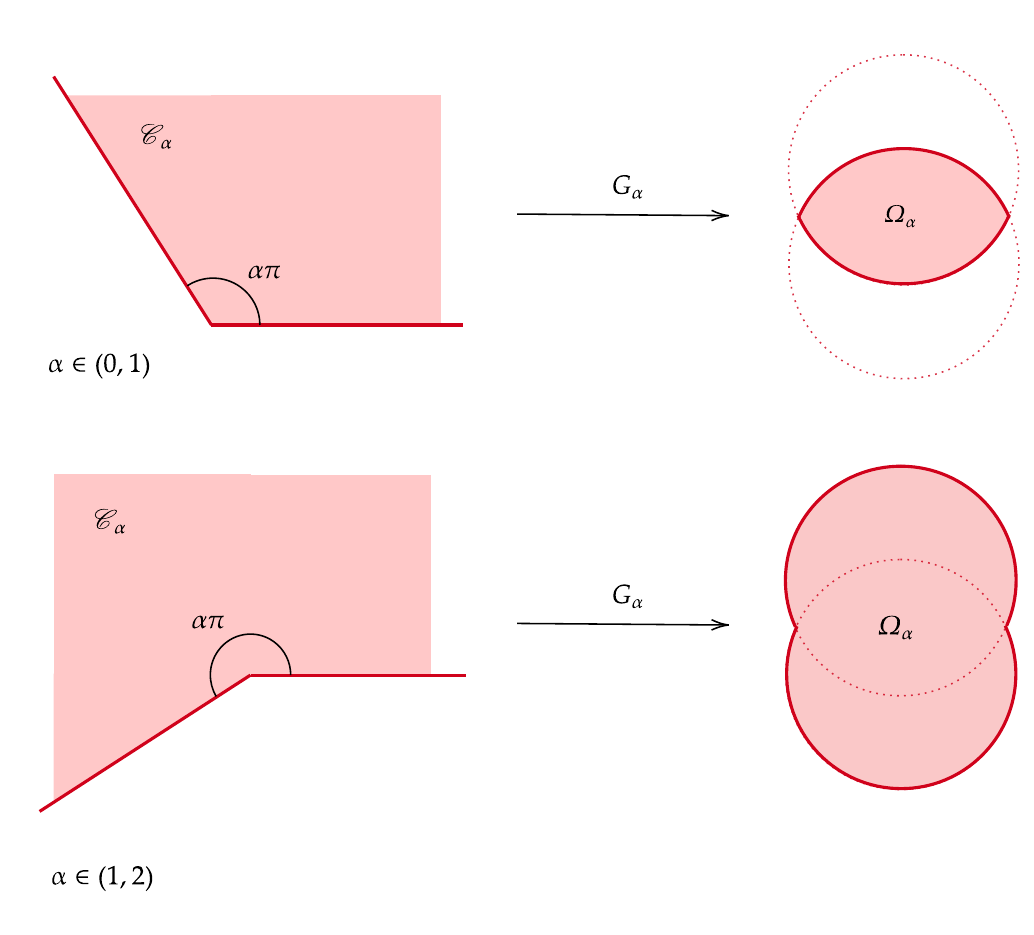}
\end{center}
\begin{proof}
First of all, we remind that
$\frac{\xi_1}{\xi_2}=\tau_\alpha$ and  $\frac{|\xi|}{|\xi_2|}=\sqrt{1+\tau_\alpha^2}.$
\\
Now, \eqref{PCConf}  follows from the complex representation of $G_\alpha$ as 
$$G_\alpha(x,y)=g(z) = -\frac{z+\xi }{z- \xi}.$$ Property \eqref{PGJacobian} follows from a direct computation: 
$$
\mathtt JG_\alpha(x,y) = |g'(z)|^2 = \frac{4|\xi|^2}{|z-\xi|^4}.
$$
To prove \eqref{cone}, first we note that if $\Pi_\alpha:=\{(x,y)\in \R^2 \;:\;  - x \sin \pi \alpha + y \cos  \pi \alpha \le 0\}$, the cone is given by
\begin{equation}\label{cone2}
\mathcal C_\alpha =
\Pi_\alpha \cap \R^2_+,\ \hbox{if}\ 0 < \alpha < 1\ \hbox{ and }\ 
\mathcal C_\alpha =\Pi_\alpha \cup \R^2_+, \ \hbox{if}\ 1 < \alpha <2. 
\end{equation}
The claim follows once we prove that
\begin{equation}\label{cone3}
G_\alpha(\R^2_+)= \mathscr D_\alpha^-\ \hbox{and}\ G_\alpha(\Pi_\alpha) =\mathscr D_\alpha^+.\end{equation}

Next, we observe that $G_\alpha$ maps the boundary of the half-spaces  $\mathbb R^2_+$ and $\Pi_\alpha$ into the boundary of the two disks of radius $\frac{|\xi|}{|\xi_2|}$
centered at the points {$\left(0,\frac{\xi_1}{\xi_2}\right)$} and  ${\left(0,-\frac{\xi_1}{\xi_2}\right)}$, respectively. Indeed
  a direct  computation shows that
$$
\left|G_\alpha(x,0)+\left(0, \frac{\xi_1}{\xi_2}\right)\right|^2 \hspace{-0.2cm}= \frac{(x+\xi_1)^2 +\xi_2^2}{(x-\xi_1)^2+\xi_2^2}-\frac{4 \xi_1 x}{(x-\xi_1)^2+\xi_2^2}+ \frac{\xi_1^2}{\xi_2^2} = 1+\frac{\xi_1^2}{\xi_2^2} = \frac{|\xi|^2}{\xi_2^2},\ \hbox{for any}\ x\in\mathbb R,
$$
 and  
$$
\left|G_\alpha(r\cos\pi \alpha , r \sin \pi \alpha) -\left(0, \frac{\xi_1}{\xi_2}\right)\right|^2 = 1+\frac{\xi_1^2}{\xi_2^2}=\frac{|\xi|^2}{\xi_2^2},\ \hbox{for any}\ r>0.
$$
Finally, we point out that   $G_\alpha$ maps the point $\xi$ at $\infty$. Moreover $\xi\not\in 
 \mathbb R^2_+$ because $\xi_2<0$ and $\xi\not\in\Pi_\alpha$ because
 $$- \xi_1 \sin \pi \alpha + \xi_2 \cos  \pi \alpha=\xi_2\left(-{1+\cos\alpha\pi\over\sin\alpha\pi}\sin\pi\alpha+\cos\pi\alpha\right)=-\xi_2>0.$$
Therefore, by \eqref{cone2} together with the fact that $G_\alpha$ maps the boundary of the half-spaces into the boundary of the disks, we deduce
\eqref{cone3}.

\end{proof}

Let  $\psi  (x,y)= \phi \left(G_\alpha^{-1}(x,y)\right)$.  Thanks to Lemma \ref{LemmaPG} we see that 
$ \phi $ solves the linear problem \eqref{lin2} if and only if $\psi$
is a solution to the Steklov problem
\begin{equation}\label{lin3}\Delta \psi  =0\ \hbox{in}\ \Omega_\alpha,\ \partial_\nu \psi= \frac{1}{\sqrt{1+\tau_\alpha^2}} \psi  \ \hbox{in}\ \partial\Omega_\alpha. \end{equation}
For the sake of completeness, let us prove the claim.
If $\phi$ solves \eqref{lin2}, for any $\upsilon\in L^{\infty}(\mathcal C_\alpha)$ satisfying \eqref{weak1},  it holds
$$\begin{aligned}
0=\int\limits_{\mathcal C_\alpha}\nabla \phi\nabla\upsilon-\int\limits_{\partial\mathcal C_\alpha}e^{W_\alpha(s,0)}\phi\upsilon.
\end{aligned}$$ 
We set $\Upsilon=\upsilon\circ G_\alpha^{-1}$. On the one hand, %we point out that 
via the change of variables $G_\alpha^{-1}(x,y)=(s,t)$ (taking into account that $G_\alpha$ is a conformal map), we have
$$\begin{aligned}
 \int\limits_{\mathcal C_\alpha}\nabla \phi(s,t)\cdot \nabla\upsilon(s,t)dsdt &= \int\limits_{\Omega_\alpha}\mathtt {det} (D G^{-1}_\alpha)(x,y)\nabla \phi\left(G_\alpha^{-1}(x,y)\right) \cdot \nabla\upsilon\left(G_\alpha^{-1}(x,y)\right)dxdy\\ &= \int\limits_{\Omega_\alpha}D G^{-1}_\alpha(x,y)\nabla \phi\left(G_\alpha^{-1}(x,y)\right)d G^{-1}_\alpha(x,y)\cdot \nabla\upsilon\left(G_\alpha^{-1}(x,y)\right)dxdy\\
 &=\int\limits_{\Omega_\alpha}\nabla \psi(x,y)\cdot \nabla\Upsilon(x,y)dxdy.
 \end{aligned}$$ 
On the other hand, we can assert that
$$\begin{aligned}
 \int\limits_{\mathcal C_\alpha}  \nabla \phi(s,t) \cdot \nabla \upsilon(s,t)dsdt
&=\int\limits_{\partial_-\mathcal C_\alpha}e^{W_\alpha(s,t)}\phi\upsilon+\int\limits_{\partial_+\mathcal C_\alpha}e^{W_\alpha(s,t)}\phi\upsilon\\
&= \frac{1}{\sqrt{1+\tau_\alpha^2}}\int\limits_{\partial\mathscr D_\alpha^-{\cap \partial \Omega_\alpha}}\psi\Upsilon+ \frac{1}{\sqrt{1+\tau_\alpha^2}}\int\limits_{\partial\mathscr D_\alpha^+ {\cap \partial \Omega_\alpha}}\psi\Upsilon\\
&= \frac{1}{\sqrt{1+\tau_\alpha^2}}\int\limits_{\partial\Omega_\alpha}\psi\Upsilon,
\end{aligned}$$
because of \eqref{jaco}, \eqref{circ}, \eqref{bordoc} and
 $$G_\alpha\left( \mathcal C_\alpha\right)=\Omega_\alpha, \ G_\alpha\left(\partial_+\mathcal C_\alpha\right)=\partial\mathscr D_\alpha^+{\cap \partial \Omega_\alpha}\ \hbox{and}\ G_\alpha\left(\partial_-\mathcal C_\alpha\right)=\partial\mathscr D_\alpha^-{\cap \partial \Omega_\alpha}.$$
Therefore,
\begin{equation}\label{EqonOmega}0= \int\limits_{ \Omega_\alpha}\nabla \psi\nabla \Upsilon-\frac{1}{\sqrt{1+\tau_\alpha^2}}\int\limits_{\partial\Omega_\alpha}\psi\Upsilon,
\end{equation}
for any $\Upsilon\in H^1(\Omega_\alpha)\cap L^\infty(\Omega_\alpha)$.
 Since $H^1(\Omega_\alpha)\cap L^\infty(\Omega_\alpha)$ is a dense subspace of $H^1(\Omega_\alpha)$, \eqref{EqonOmega} holds for any $\Upsilon\in H^1(\Omega_\alpha)$, namely $\psi$ solves \eqref{lin3}.
\subsection{Proof of Theorem \ref{main}: conclusion}\label{conclusion}
It is immediate to check that $\mu_\alpha:= \frac{1}{\sqrt{1+\tau_\alpha^2}}$ is an eigenfunction of the Steklov problem \eqref{lin3}
and also that 
\begin{equation*}%\label{psi0}
	 \psi(x,y)=x,\end{equation*}
is an associate eigenfunction.

In Section \ref{Sec:Steklov}  we prove that $\mu_\alpha$ is simple for all $\alpha\in(0,1)\cup(1,2)$. Thus, for {$\alpha\in (0,1)\cup(1,2)$}, $\psi$ is the unique eigenfunction corresponding to $\mu_\alpha$ and, using all the previous arguments, we deduce that all the solutions to \eqref{Lin} are a scalar multiple of the function
$$
\Phi(x,y) =    (\psi \circ G_\alpha\circ F_\alpha) (x,y) = \frac{|\xi|^2- (x^2+y^2)^\alpha}{|(x+iy)^\alpha - \xi_1-i \xi_2|^2},
$$
concluding the proof of Theorem \ref{main}.

 {Note that the conditions $\frac{\xi_1}{\xi_2}=\frac{1+\cos(\alpha\pi)}{\sin(\alpha\pi)}$ and $\xi_2<0$ that we have used in this Section are equivalent to the conditions given in \eqref{bubble} with $z_0=\xi$ and $\rho=|\xi|$. Thus $\Phi$ reads as \eqref{eq:Phi_cond} with $\mathfrak{c}=-\frac{1}{\rho}$.	}

\section{On the simplicity of the eigenvalue of the Steklov problem}\label{Sec:Steklov}
This last section is devoted to the study of the Steklov eigenvalue problem \eqref{lin3} and, in particular, to proving the simplicity of a distinguished eigenvalue, which will allow us to conclude the proof of Theorem \ref{main}, as anticipated in Section \ref{conclusion}.

Let $\Omega$ be a bounded Lipschitz domain in $\mathbb R^n$ and denote by $\{\mu_k\}_{k=1}^{\infty}$ its Steklov eigenvalues. Then, the spectrum of the homothetic domain $c\Omega$, $c>0$, is given by $\{c^{-1}\mu_k\}_{k=1}^{\infty}$. Applying this observation to $c\Omega_{\alpha}$ with $c^2=\frac{1}{1+\tau_{\alpha}^2}$ we conclude that the simplicity of the eigenvalue $\frac{1}{\sqrt{1+\tau_{\alpha}^2}}$ on $\Omega_{\alpha}$ for $\alpha\in(0,1)\cup(1,2)$ is equivalent to the simplicity of the eigenvalue $1$ on the following two families of domains:
$$
\mathscr D_{\ell}^+\cap\mathscr D_{\ell}^-{\rm\ \ \ and\ \ \ }\mathscr D_{\ell}^+\cup\mathscr D_{\ell}^-,
$$
 where $\mathscr D_{\ell}^{\pm}=\{(x,y)\in\mathbb R^2:x^2+(y\pm\ell)^2<1\}$ and $\ell\in(0,1)$. In fact, the function $\psi=x$ is always an eigenfunction with eigenvalue $1$ for any domain of each family.

\subsection{The eigenvalue $1$ of the intersection}\label{sub:intersection}

Let  $\Omega_{\ell}:=\mathscr D_{\ell}^+\cap\mathscr D_{\ell}^-$, $\ell\in(0,1)$ (see Figure \ref{figM1}, left). We consider the Steklov eigenvalues on $\Omega_{\ell}$ and denote them by $\mu_k(\ell)$.  As usual, $\mu_1(\ell)=0$ with constant eigenfunctions. We  prove that in this situation $1$ is the second eigenvalue and it is simple for all $\ell\in(0,1)$.

\begin{figure}
\includegraphics[width=0.8\textwidth]{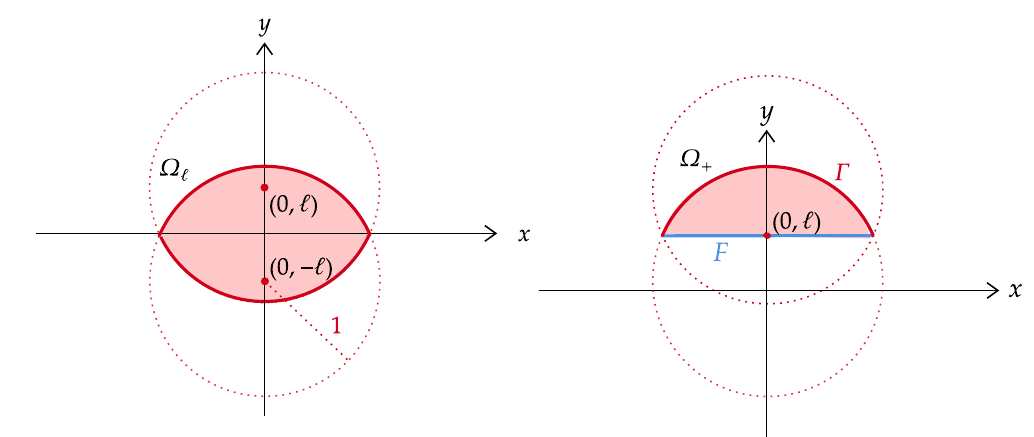}
\caption{}
\label{figM1}
\end{figure}

\begin{proposition}\label{prop_intersection}
For any $\ell\in(0,1)$, $\mu_2(\ell)=1$ and it is simple.
\end{proposition}
\begin{proof}
Since $\Omega_{\ell}$ is symmetric with respect $x$ and $y$, any Steklov eigenfunction $\psi$ is written as $\psi=\psi_{ee}+\psi_{oo}+\psi_{eo}+\psi_{oe}$, where $\psi_{ee}$ is even with respect to $x$ and $y$, $\psi_{oo}$ is odd with respect to $x$ and $y$, $\psi_{eo}$ is even with respect to $x$ and odd with respect to $y$ and $\psi_{oe}$ is odd with respect to $x$ and even with respect to $y$. In other words, any eigenspace is spanned by eigenfunctions of these types.  Now, any second eigenfunction $\psi$ on $\Omega_{\ell}$ has exactly two nodal domains by the classical Courant's theorem for the Steklov problem, see e.g., \cite{ks1}. If $\psi$ is a second eigenfunction, it cannot be of the form $\psi_{oo}$, because otherwise it would have  at least four nodal domains. It cannot be of the form $\psi_{ee}$. In fact, it cannot have interior nodal domains, since it is harmonic. This implies that {there} are at least three nodal domains, which is not possible. 

{Suppose that} $x$ is not a second eigenfunction and there {exists} a second eigenfunction {$\psi_{oe}$}, then $\int_{\partial\Omega_{\ell}}\psi_{oe}x=0$ (eigenfunctions corresponding to different eigenvalues are orthogonal in $L^2(\partial\Omega_{\ell})$). This is not possible since $\psi_{oe}$ has two nodal domains: 
$\Omega_{\ell}\cap\{(x,y)\in\mathbb R^2:x>0\}$ and 
 $\Omega_{\ell}\cap\{(x,y)\in\mathbb R^2:x<0\}$. For the same reason, if $x$  is a second eigenfunction, there are no other $\psi_{oe}$ second eigenfunctions. The same reasoning tells us that we can have only one linearly independent second eigenfunction of type $\psi_{eo}$. Altogether, we have proved that $\mu_2(\ell)$ has multiplicity at most $2$, and that the only three possibilities are that it is simple with eigenfunction $x$, it is simple with eigenfunction $\psi_{eo}$, or it is double with eigenspace  spanned by $\{x,\psi_{eo}\}$ for some eigenfunction $\psi_{eo}$.

%{\color{gab} Some comments here: I believe the remaining part of the proof is correct but perhaps it is simpler to obtain the result by testing the equation against the test function $v(x,y) = y+l$. We could write something like: 
%
%To conclude the proof we show that there is no second eigenfunction $\psi_{eo}$. Let $\Omega_+:=\Omega_{\ell}\cap\{y>0\}$. Note that $\partial\Omega_+=\Gamma\cup F$, where $\Gamma=\partial\Omega_+\cap\{(x,y)\in\partial \Omega \;:\; y>0\}$  and $F=\partial\Omega_+\cap\{(x,y)\in\mathbb R^2:y=0\}$. We consider the function $v(x,y) = y + \ell$ which is harmonic in $\Omega_+$ and satisfies $\partial_\nu v = v$ on $\Gamma$. Then
%$$
%0 = \int_{\partial \Omega_+} \partial_\nu \psi_{eo}\, v - \psi_{eo} \partial_\nu v  = (\mu_2(\ell) - 1) \int_{\Gamma} \psi_{eo} v + \ell \int_{F} \partial_\nu \psi_{eo}
%$$
%Without loss of generality, we may assume that $\psi_{eo}>0 $ in $\Omega^+$. By Hopf lemma, we also have $\partial_\nu \psi_{eo} < 0$ on $F$. Since $\mu_2(\ell)\le 1$, we get that 
%$$
%(\mu_2(\ell) - 1) \int_{\Gamma} \psi_{eo} v + \ell \int_{F} \partial_\nu \psi_{eo}<0,
%$$
%which is a contradiction. Is this argument correct? 
%
%
%
%
% } 

 To conclude the proof, we show that the Rayleigh quotient of any eigenfunction $\psi_{eo}$  is strictly larger than $1$. 
 To do so, we compare the Steklov eigenvalue associated with some eigenfunction $\psi_{eo}$ with an eigenvalue of the Laplacian on the boundary by following the strategy of \cite{PS}. This is achieved by using a Rellich-Pohozaev identity.

 More precisely, let $\Omega_+:=\Omega_{\ell}\cap\{y>0\}$. It is convenient to  translate the origin of the coordinate system in $(0,-\ell)$ (see Figure \ref{figM1}, right). Let $\psi_{eo}$ be an even-odd eigenfunction with eigenvalue $\mu$. Let $p=(x,y)$ be the position vector in the new coordinate system (see Figure \ref{figM1}, right). In the new coordinate system $\partial\Omega_+=\overline\Gamma\cup\overline F$, where $\Gamma=\partial\Omega_+\cap\{(x,y)\in\mathbb R^2:y>\ell\}$  and $F=\partial\Omega_+\cap\{(x,y)\in\mathbb R^2:y=\ell\}$. If $\nu$ denotes the outer unit normal to $\Omega_+$ we have that $\nu=p$ on $\Gamma$ and $\nu=(0,-1)$ on $F$.  Moreover { since $\psi_{eo}=0$ on $F$% by oddness in $y$
 	, we have }$\partial_x\psi_{eo}=0$ on $F$. 
To simplify the notation, in the following formulas we write  $\psi$ for $\psi_{eo}$. We have

\begin{multline}\label{rellich1}
	0=\int_{\Omega_+}\Delta \psi\,p\cdot\nabla \psi=\int_{\Gamma}\partial_{\nu}\psi\,p\cdot \nabla \psi+\int_{F}\partial_{\nu}\psi\,p\cdot \nabla \psi-\int_{\Omega_+}\nabla \psi\cdot\nabla(p\cdot\nabla \psi)\\
	=\mu^2\int_{\Gamma}\psi^2-\ell\int_F(\partial_y\psi)^2-\int_{\Omega_+}|\nabla \psi|^2-\frac{1}{2}\int_{\Omega_+}\nabla(|\nabla \psi|^2)\cdot p\\
	=\mu^2\int_{\Gamma}\psi^2-\ell\int_F(\partial_y\psi)^2-\frac{1}{2}\int_{\Gamma}|\nabla \psi|^2(p\cdot\nu)-\frac{1}{2}\int_{F}|\nabla \psi|^2(p\cdot\nu)\\ 
	=\mu^2\int_{\Gamma}\psi^2-\ell\int_F(\partial_y\psi)^2-\frac{\mu^2}{2}\int_{\Gamma}\psi^2-\frac{1}{2}\int_{\Gamma}|\nabla_{\Gamma}\psi|^2+\frac{\ell}{2}\int_F(\partial_y\psi)^2.
\end{multline}
From \eqref{rellich1} we deduce
$$
\mu^2\geq \frac{\int_{\Gamma}|\nabla_{\Gamma}\psi|^2}{\int_{\Gamma}\psi^2}.
$$
Here $\nabla_{\Gamma}\psi$ stands for the tangential component of $\nabla\psi$ along $\Gamma$. Since $\psi=0$ at the endpoints of $\Gamma$, we have that
$$
\mu^2\geq \frac{\int_{\Gamma}|\nabla_{\Gamma}\psi|^2}{\int_{\Gamma}\psi^2}\geq\min_{0\ne u\in H^1_0(\Gamma)}\frac{\int_{\Gamma}|\nabla_{\Gamma}u|^2}{\int_{\Gamma}u^2}=\lambda_1^D(\Gamma)=\frac{\pi^2}{|\Gamma|^2}>1,
$$
where $\lambda_1^D(\Gamma)$ is the first Dirichlet eigenvalue on $\Gamma$ and $|\Gamma|<\pi$ is the length of $\Gamma$. This concludes the proof.
\end{proof}

\subsection{The eigenvalue $1$ of the union}\label{sub:union}

Let $\Omega_{\ell}:=\mathscr D_{\ell}^+\cup\mathscr D_{\ell}^-$ (see Figure \ref{fig0}, left). As before, by $\mu_k(\ell)$ we denote the Steklov eigenvalues of $\Omega_{\ell}$. We  prove that in this situation $1$ is the third eigenvalue and it is simple for all $\ell\in(0,1)$.
\begin{figure}
\includegraphics[width=0.7\textwidth]{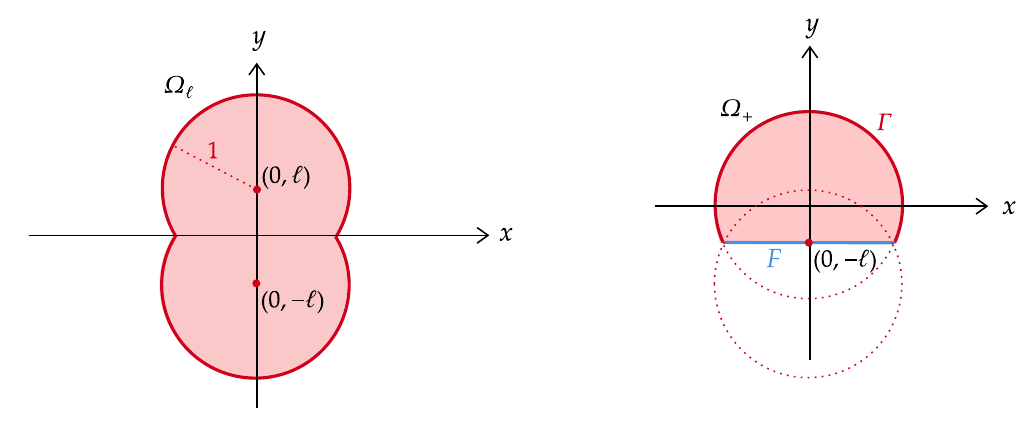}
\caption{}
\label{fig0}
\end{figure}
\begin{proposition}\label{prop_union}
For any $\ell\in(0,1)$:
\begin{enumerate}[i)]
\item $\mu_2(\ell)<1$ and it is simple;
\item $\mu_3(\ell)=1$ for all $\ell\in(0,1)$ and it is simple.
\end{enumerate}
\end{proposition}
\begin{proof}
We first prove $i)$. We recall Weinstock's inequality \cite{weinstock}: for any simply connected domain $\Omega$ of $\mathbb R^2$ the second Steklov eigenvalue $\mu_2$ satisfies
$$
\mu_2\leq\frac{2\pi}{|\partial\Omega|}
$$
and the equality holds if and only if $\Omega$ is a disk. In our case we have
$$
\mu_2(\ell)\leq\frac{2\pi}{|\partial\Omega_{\ell}|}<1
$$
since $|\partial\Omega_{\ell}|>2\pi$.
Next, we prove that the second eigenvalue is simple and a second eigenfunction is odd with respect to $y$. Let $\psi$ be a second eigenfunction. As in Subsection \ref{sub:intersection}, we see that
$\psi=\psi_{ee}+\psi_{oo}+\psi_{eo}+\psi_{oe}$.
Moreover, $\psi$ has exactly two interior nodal domains, hence we cannot have $\psi_{ee}$ and $\psi_{oo}$ second eigenfunctions. We observe that we cannot have $\psi_{oe}$ second eigenfunctions either. In fact, $x$ is a Steklov eigenfunction with eigenvalue $1$, and it is not a second Steklov eigenfunction. Then $x$ is orthogonal in $L^2(\partial\Omega_{\ell})$ to a second eigenfunction (and to the first eigenfunction, the constant). If we have a $\psi_{oe}$ second eigenfunction, then $\int_{\partial\Omega_{\ell}}\psi_{oe}x=0$. But a second eigenfunction of type $\psi_{oe}$ has two nodal domains: $\Omega_{\ell}\cap\{(x,y)\in\mathbb R^2:x>0\}$ and $\Omega_{\ell}\cap\{(x,y)\in\mathbb R^2:x<0\}$, hence $\int_{\partial\Omega_{\ell}}\psi_{oe}x\ne 0$. We conclude that a second eigenfunction can be only of type $\psi_{eo}$ and that it is simple. Indeed, following the same argument as above, we see that we cannot have more than one linearly independent $\psi_{eo}$ eigenfunction with exactly two nodal domains. This concludes the proof of $i)$.

\medskip

We now prove $ii)$. The proof is divided in three main steps.\\
{\bf Step 1.} The eigenvalues behave continuously for $\ell\in(0,1)$ (see e.g., \cite{Bogosel,Bucur_steklov,FL_steklov}), and in particular, as $\ell\to 0$, they converge to the Steklov eigenvalues of the unit disk:
$$
0,1,1,2,2,3,3,\cdots
$$
The proof follows from classical spectral stability results for the Steklov problem, {see \cite[Theorem $3.5$]{Bogosel}, see also \cite{Bucur_steklov,FL_steklov}}. Precisely, $\mu_1(\ell)=0$ for all $\ell$, $\mu_2(\ell)\to 1$, $\mu_3(\ell)\to 1$ and $\mu_4(\ell)\to 2$ as $\ell\to 0$.
This implies that $ii)$ holds for $\ell\in(0,\ell_0)$ for some $\ell_0\in(0,1)$. 
%{\color{blue}This implies that there exists $\ell_0\in(0,1)$ such that $ii)$ holds for $\ell\in(0,\ell_0)$.??}{\color{gab} (Yes, because we know that for small $\ell$, $\mu_3(\ell)$ is close to $1$ and $\mu_4(l)$ is close to $2$. Since we proved $i)$, we know that $0<\mu_2 <\mu_3 <1+\eps < 2-\eps < \mu_4$ for $\ell$ small enough. Since $1$ is an eigenvalue and $\mu_2<1$, we must have $\mu_3 =1$.)}

Let $\psi$ be a  Steklov eigenfunction. It can be written as $\psi=\psi_{ee}+\psi_{oo}+\psi_{eo}+\psi_{oe}$. The eigenfunction $x$ corresponding to $\mu=1$ is of type $\psi_{oe}$. In the next steps we show that there are no other eigenfunctions  $\psi_{oe}$ in the eigenspace of $\mu=1$ for all $\ell\in(0,1)$, and that there are no eigenfunctions of type $\psi_{ee},\psi_{oo},\psi_{oe}$ in the eigenspace of $\mu=1$ for all $\ell\in(0,1)$. This fact and the continuity of the eigenvalues imply the validity of $ii)$. 

{\bf Step 2.} We prove that any $\psi_{ee}$ eigenfunction has eigenvalue strictly larger than $1$. To do so we follow the same strategy of Subsection \ref{sub:intersection} using a Rellich-Pohozaev identity as in \cite{PS}. More precisely, let $\Omega_+:=\Omega_{\ell}\cap\{(x,y)\in\mathbb R^2:y>0\}$. The restriction of $\psi_{ee}$ to $\Omega_+$ satisfies $-\partial_y\psi_{ee}=0$ at $y=0$ due to parity.
As in Subsection \ref{sub:union}, it is convenient to translate the origin of the coordinate system in $(0,\ell)$ (see Figure \ref{fig0}, right).
In this new system of coordinates, let $p=(x,y)$ be the position vector. Note that $\partial\Omega_+=\overline\Gamma\cup \overline F$ where $\Gamma=\partial\Omega_+\cap\{(x,y)\in\mathbb R^2:y>{-\ell}\}$  and $F=\{\partial\Omega_{\ell}\cap y={-\ell}\}$. Also, if $\nu$ is the outer unit normal to $\partial\Omega_+$, we have $\nu=p$ on $\Gamma$ and $\nu=(0,-1)$ on $F$. In particular, $\partial_{\nu}\psi_{ee}=-\partial_y\psi_{ee}=0$ on $F$. In the following formulas we write $\psi$ for $\psi_{ee}$:
\begin{multline}
0=\int_{\Omega_+}\Delta \psi\,p\cdot\nabla \psi=\int_{\Gamma}\partial_{\nu}\psi\,p\cdot \nabla \psi-\int_{\Omega_+}\nabla \psi\cdot\nabla(p\cdot\nabla \psi)\\
=\mu^2\int_{\Gamma}\psi^2-\int_{\Omega_+}|\nabla \psi|^2-\frac{1}{2}\int_{\Omega_+}\nabla(|\nabla \psi|^2)\cdot p\\
=\mu^2\int_{\Gamma}\psi^2-\frac{1}{2}\int_{\Gamma}|\nabla \psi|^2p\cdot\nu-\frac{1}{2}\int_{F}|\nabla \psi|^2p\cdot\nu\\
=\frac{\mu^2}{2}\int_{\Gamma}\psi^2-\frac{1}{2}\int_{\Gamma}|\nabla_{\Gamma}\psi|^2-\frac{\ell}{2}\int_{F}|\nabla_{\Gamma}\psi|^2.
\end{multline}
We conclude that
\begin{equation}\label{ineq0}
\mu^2\geq\frac{\int_{\Gamma}|\nabla_{\Gamma}\psi|^2}{\int_{\Gamma}\psi^2}.
\end{equation}
Now, $\psi$ is even with respect to $y$ and $\int_{\partial\Omega_{\ell}}\psi=0$, hence $\int_{\Gamma}\psi=0$. Moreover, the second Neumann eigenfunction $v$ on the curve $\Gamma$ is odd with respect to $x$ (in the arclength variable $s$ on $\Gamma$, $v(s)=\cos(\pi s/|\Gamma|$). Hence $\int_{\Gamma}\psi v=0$. Thus
$$
\mu^2\geq\frac{\int_{\Gamma}|\nabla_{\Gamma}\psi|^2}{\int_{\Gamma}\psi^2}\geq\min_{{\substack{0\ne u\in H^1(\Gamma)\\\int_{\Gamma}u=\int_{\Gamma}uv=0}}}\frac{\int_{\Gamma}|\nabla_{\Gamma}u|^2}{\int_{\Gamma}u^2}=\lambda^N_3(\Gamma)=\frac{4\pi^2}{|\Gamma|^2}>1
$$
where $\lambda_3^N(\Gamma)$ is the third Neumann eigenvalue on $\Gamma$ and $|\Gamma|<2\pi$ is the length of $\Gamma$. In conclusion, there are no $\psi_{ee}$ second eigenfunctions.

{\bf Step 3.} Assume  that there exists $\ell_0\in(0,1)$ such that $\mu_3(\ell)=1$ and it is simple for all $\ell\in(0,\ell_0)$, while $\mu_3(\ell_0)=1$ is not simple.  Recall that any eigenfunction associated to $\mu_3(\ell_0)$ has at most three nodal domains from \cite{ks1}. {Suppose that $\mu_3(\ell_0)$ has an eigenfunction $\psi_{oe}$ which is $L^2$ orthogonal to $x$}. Clearly $\psi_{oe}(0,y)=0$ but $x=0$  cannot be the unique nodal line since $\int_{\partial\Omega_{\ell_0}}\psi_{oe}x=0$. Hence we have another nodal line in $\Omega_{\ell_0}\cap\{(x,y)\in\mathbb R^2:x>0\}$ which has non-empty intersection with $\partial\Omega_{\ell_0}$. Hence there are at least two nodal domains in $\Omega_{\ell_0}\cap\{(x,y)\in\mathbb R^2:x>0\}$, but since the eigenfunction is odd in $y$, there are at least two nodal domains also in $\Omega_{\ell_0}\cap\{(x,y)\in\mathbb R^2:x<0\}$, and  this is not possible.

Suppose that $\mu_3(\ell_0)$ has a  $\psi_{eo}$ eigenfunction. Then $\psi_{eo}(x,0)=0$, but since $\psi_{eo}$ is orthogonal to the second eigenfunction in $L^2(\partial\Omega_{\ell_0})$ (which has $y=0$ as unique nodal line), following the same argument above we conclude that there must exist at least other two nodal lines, which are symmetric with respect to the $y$ axis, hence at least four nodal domains, which is not possible.

Finally, suppose that $\mu_3(\ell_0)$ has a $\psi_{oo}$ eigenfunction. Then $x=0$ and $y=0$ are nodal lines. Hence there are at least four nodal domains which is not possible.

\end{proof}

\section*{Acknowledgements} {We thank the anonymous referee for their valuable comments.} The first three authors are partially supported by the group GNAMPA of the Istituto Nazionale di Alta Matematica (INdAM).  In particular, they acknowledge financial support from INdAM-GNAMPA Project 2023, codice CUP E53C2200193000; Project 2024, codice CUP E53C23001670001; Project 2025, codice CUP E5324001950001 and Project 2026, codice CUP E53C25002010001. 
A. DlT. acknowledges financial support from the Spanish Ministry of Science and Innovation (MICINN), through the IMAG-Maria de Maeztu Excellence Grant CEX2020-001105-M/AEI/10.13039/501100011033. She is also supported by the FEDER-MINECO Grants PID2021- 122122NB-I00, PID2020-113596GB-I00 and PID2024-155314NB-I00; RED2022-134784-T, funded by MCIN/AEI/10.13039/501100011033 and by J. Andalucia (FQM-116); Fondi Ateneo - Sapienza Università di Roma; and DFG - Projektnummer: 561401741.
A. DlT. and A. P. are supported by PRIN (Prot. 20227HX33Z). G. M. is supported the PRIN PNRR Project  P2022YFAJH {\sl \lq\lq Linear and Nonlinear PDEs: New directions and applications''} and the PNRR MUR project CN00000013 HUB - National Centre for HPC, Big Data and Quantum Computing (CUP H93C22000450007).
L. P. is member of the group GNSAGA of the Istituto Nazionale di Alta Matematica (INdAM). L. P. acknowledges support from the project {\sl \lq\lq Perturbation problems and asymptotics for elliptic differential equations: variational and potential theoretic methods''} funded by the European Union - Next Generation EU and by MUR Progetti di Ricerca di Rilevante Interesse Nazionale (PRIN) Bando 2022 grant 2022SENJZ3.

\bibliography{biblio}
\bibliographystyle{abbrv}

\end{document}